\documentclass[a4paper,11pt]{amsart}

\usepackage{amsmath}
\usepackage{amssymb}
\usepackage{amsthm}
\usepackage{a4wide}
\usepackage{verbatim}
\usepackage[all]{xy}
\usepackage{enumerate}
\usepackage{dsfont}
\usepackage{color}

\usepackage{graphicx}

\DeclareGraphicsExtensions{.png}

\newtheorem{thm}{Theorem}[section]
\newtheorem{prop}[thm]{Proposition}
\newtheorem{cor}[thm]{Corollary}
\newtheorem{lem}[thm]{Lemma}

\theoremstyle{definition}

\newtheorem{dfn}[thm]{Definition}

\newtheorem{rmk}[thm]{Remark}

\numberwithin{equation}{section}
\newcommand{\spa}{\textrm{span}}
\newcommand{\cM}{\mathsf{M}}
\newcommand{\cH}{\mathcal{H}}

\newcommand{\Id}{\textrm{Id}}

\newcommand{\N}{\mathbb{N}}

\newcommand{\C}{\mathbb{C}}
\newcommand{\Z}{\mathbb{Z}}

\newcommand{\wt}{\widetilde}

\newcommand{\cB}{\mathsf{B}}

\newcommand{\VN}{\textup{VN}}

\newcommand{\cHR}{\mathcal{H}_{\mathbb{R}}}

\newcommand{\cFq}{\mathcal{F}_q}

\newcommand{\exE}{\mathbb{E}}

\newcommand{\Hecke}{\C_q[W]}
\newcommand{\VNHecke}{\VN_q(W)}

\newcommand{\Heckel}{\C_q^l[W]}
\newcommand{\HeckeOne}{\C_q^1[W]}

\newcommand{\Heckelprime}{\C_q^{l'}[W]}

\newcommand{\cA}{\mathsf{A}}

\title{On MASAs in $q$-deformed von Neumann algebras}

\author{Martijn Caspers}
\address{Mathematisch Instituut, Universiteit Utrecht,
	Budapestlaan 6, 3584 CD Utrecht,
	The Netherlands}
\email{m.p.t.caspers@uu.nl}

\author{Adam Skalski}
\address{Institute of Mathematics of the Polish Academy of Sciences, ul. \'{S}niadeckich 8, 00-656 Warsaw, Poland}
\email{a.skalski@impan.pl}

\author{ Mateusz Wasilewski}
\address{Institute of Mathematics of the Polish Academy of Sciences,  ul. \'{S}niadeckich 8, 00-656 Warsaw, Poland}
\email{mwasilewski@impan.pl}

\begin{document}
\begin{abstract}
We study certain $q$-deformed analogues of the maximal abelian subalgebras of the group von Neumann algebras of free groups.
The radial subalgebra is defined for Hecke deformed von Neumann algebras of the Coxeter group $(\mathbb{Z}/{2\Z})^{\star k}$ and shown to be a maximal abelian subalgebra which is singular and with Puk\'anszky invariant $\{\infty\}$. Further all non-equal generator masas in the $q$-deformed Gaussian von Neumann algebras are shown to be mutually non-unitarily conjugate.
\end{abstract}

\subjclass[2010]{Primary: 46L10, Secondary: 46L65}

\keywords{maximal abelian subalgebras; singular masas; Hecke von Neumann algebra; q-Gaussian algebras}

\maketitle

\section{Introduction}

Our aim is to investigate maximal abelian subalgebras in certain ${\rm II}_1$-factors that can be viewed as deformations of $\VN(\mathbb{F}_n)$. Our particular interest lies in the analysis of counterparts of the radial masa $A_r$ in $\VN(\mathbb{F}_n)$, studied for example in \cite{BocaRad} and in \cite{Stuart} (see also \cite{Trenholme}). The main open problem concerning the radial masa in
$\VN(\mathbb{F}_n)$ is the question whether it is isomorphic to the generator masa(s); so far they share all the known properties, such as maximal injectivity, same Puk\'anszky invariant, etc. They are also known not to be unitarily conjugate (see Proposition 3.1 of \cite{Stuart}). More generally, radial masas have been studied for von Neumann algebras of groups of the type $(\Z/_{n\Z})^{\star k}$ in \cite{Trenholme} and \cite{BocaRad}.

Here we want to analyse the behaviour of counterparts of the radial/generator masa in some deformed versions of $\VN(\mathbb{F}_n)$ or $\VN((\Z/_{n\Z})^{\star k})$; more specifically in Hecke deformed von Neumann algebras of right-angled Coxeter groups $\VNHecke$ of  Dymara (\cite{Dymara}, see also \cite{Garncarek} and \cite{Caspers}) and in $q$-deformed Gaussian von Neumann algebras $\Gamma_q(\cHR)$ of Bo\.zejko, K\"ummerer and Speicher (\cite{BozejkoSpeicher}). In the former case we can naturally define the radial subalgebra (and not the generator one), and in the latter the object that intuitively corresponds to the radial subalgebra is in fact obviously isomorphic to the generator one (as studied by Ricard in \cite{ricard05qfactor} and further by Wen in \cite{Wen} and Parekh, Shimada and Wen in \cite{qMaxInjective}). We show in Section \ref{SectionqGaussian} however that the different generator masas inside the $\Gamma_q(\cHR)$ are not unitarily conjugate.

Note that another example of a counterpart of the radial subalgebra in $\VN(\mathbb{F}_n)$ was studied and shown to be maximal abelian and singular in \cite{AmauryRoland}. It was a von Neumann subalgebra of the algebra $L^{\infty}(O_N^+)$, which shares many properties with $\VN(\mathbb{F}_n)$, although very recently was shown to be non-isomorphic to the latter in \cite{MikeRoland}.

The plan of the paper is as follows: after finishing this section with introducing certain notations, in Section 2 we define the radial subalgebra of the Hecke deformed von Neumann algebra $\VNHecke$ and show it to be maximal abelian. In Section 3 we compute its Puk\'anszky invariant and deduce its singularity. Finally Section 4 discusses the non-unitary-conjugacy of a (continuous family of) different generator masas in the $q$-deformed Gaussian von Neumann algebras.


\vspace{0.3cm}

\noindent {\bf Notation.} Throughout this paper by a {\bf masa} we mean a maximal abelian von Neumann subalgebra of a given von Neumann algebra $\cM$.
Let $U(\cM)$ be the group of unitaries in $\cM$.
For a (unital) subalgebra $\cA \subseteq \cM$ we define the {\bf normalizer} of $\cA$ in $\cM$ as
\[
N_{\cM}(\cA) = \{ u \in U(\cM) \mid u \cA u^\ast \subseteq \cA \}.
\]
A subalgebra $\cA \subseteq \cM$ is called {\bf singular} if $N_{\cM}(\cA) \subseteq \cA$.

$\N_0$ denotes the natural numbers including 0.

\section{The radial Hecke masa}\label{Sect=Masa}
In this section we show that right angled Hecke von Neumann algebras admit a radial algebra and prove that it is in fact a masa.

Let $W$ denote a {\bf right-angled Coxeter group}. Recall that this is the universal group generated by a finite set $S$ of elements of order 2, with the relations forcing some of the distinct elements of $S$ to commute, and some other to be free. This is formally encoded by a function $m:S \times S \setminus \{(s,s):s \in S\} \to \{2, \infty\}$ such that for all $s,t \in S, s \neq t$ we have
\[(st)^{m(s,t)} = e\]
(and $(st)^{\infty}=e$ means that $s$ and $t$ are free; necessarily $m(s,t) = m(t,s)$). We will always associate to $W$ the length function $|\cdot|:W \to \N_0$ given by the generating set $S$. All the information about $W$ is encoded by a graph $\Gamma$ with a vertex set $V\Gamma= S$ and the edge set $E\Gamma=\{(s,t) \in S \times S: m(s,t)=2\}$. Let $q\in(0, 1]$ and put $ p = \frac{q-1}{q^\frac{1}{2}}$ (note that our convention on $q$ means that $p\leq 0$).  The algebra $\Hecke$ is a *-algebra with a linear basis $\{T_w: w \in W\}$
satisfying the conditions ($s \in S, w \in W$)
\[ T_s T_w = \begin{cases} T_{sw} & \textup{ if }  |sw| > |w|, \\
T_{sw} + p T_w &  \textup{ if }  |sw| < |w|.
\end{cases}\]
The algebra $\Hecke$ acts in a natural way (via bounded operators) on the space $\ell^2(W)$ and its von Neumann algebraic closure in $B(\ell^2(W))$ will be denoted by $\VNHecke$. The vector $\delta_e \in \ell^2(W)$ will sometimes be denoted by $\Omega$; the corresponding vector state $\tau:=\omega_\Omega$ on $\VNHecke$ is a faithful trace. More generally to any element $T \in \VNHecke$ we can associate its {\bf symbol} $T\Omega$, and as $\Omega$ is a separating vector for $\VNHecke$ this correspondence is injective. Finally note that using the right action of the Hecke algebra on itself we can define another von Neumann algebra acting on  $\ell^2(W)$, say $\VNHecke^r$. It is obviously contained in the commutant of $\VNHecke$; in fact Proposition 19.2.1 of \cite{Davisbook} identifies it with $\VNHecke'$.

We will write in what follows $L$ to denote the cardinality of $S$.

\vspace{0.3cm}

Hecke von Neumann algebras were first considered in \cite{Dymara} and \cite{DymaraEtAl} in order to study weighted $L^2$-cohomology of Coxeter groups. In \cite{DymaraEtAl} the authors raised a natural question: how large is the centre of $\VNHecke$? A precise answer for the right-angled case was found   in \cite{Garncarek}, where Garncarek showed the following result.

\begin{thm}\label{Thm=Factoriality}
Let $\vert S \vert \geq 3$ and assume that $\Gamma$ is irreducible. Then for $q \in [ \rho, 1]$ the right-angled Hecke von Neumann algebra $\Hecke$ is a ${\rm II}_1$ factor and for $(0, \rho)$ we have that $\Hecke$ is a direct sum of a ${\rm II}_1$-factor and $\mathbb{C}$. Here $\rho$ is the radius of convergence of the fundamental power series $\sum_{k =0}^\infty \vert \{ w \in W \mid \vert w \vert = k \} \vert z^k$.
\end{thm}

In particular $\VNHecke$ is diffuse if and  only if $q \in [\rho, 1]$. Further   structural results were obtained in \cite{Caspers},  \cite{CaspersConnes}, \cite{CaspersFima} where for example non-injectivity, approximation properties, absence of Cartan subalgebras, the Connes embedding property and the existence of graph product decompositions were established for $\VNHecke$.

\vspace{0.3cm}

In this paper we consider the special case $W=(\Z_2)^{*L}$, i.e.\ the case where $m$ is constantly equal infinity. We assume also that $L\geq 3$. Here the main result of \cite{Garncarek}, c.f. Theorem \ref{Thm=Factoriality}, says that $\VNHecke$ is a factor if and only if $q \in [\frac{1}{L-1},1]$, and results of  \cite{Dykema} together with a calculation in Section 5 of \cite{Garncarek} show that for that range of $q$ we have $\VNHecke \approx \VN(\mathbb{F}_{\frac{2Lq}{(1+q)^2}})$, where $\VN(\mathbb{F}_s)$ for $s\geq 1$ denote the interpolated free group factors of Dykema and Radulescu.

\begin{dfn}
An element $T \in \VNHecke$ is said to be radial if  for its symbol decomposition $T\Omega= \sum_{w \in W} c_w \delta_w$, where $c_w \in \C$ we have $c_w = c_v$ for every $v,w \in W$ with $l(v) = l(w)$. We say that $T$ has radius (at most) $n$ if the frequency support (i.e.\ the set of those $w \in W$ for which $c_w\neq 0$) of $T_w$ is contained in the ball $\{ w \in W: |w| \leq n \}$.
\end{dfn}

Define $h \in \Hecke\subset \VNHecke$ by the formula $h = \sum_{s \in S} T_s$ and put $\cB:= \{h\}''$.

\begin{prop}\label{Prop=Radial}
The von Neumann algebra $\cB$  coincides with the collection of all radial operators in  $\VNHecke$. In particular the set of all radial operators forms an algebra.
\end{prop}
\begin{proof}
For each $n \in \N$ consider the radial operator $h_n:=\sum_{w \in W, \vert w \vert =n}  T_w \in \Hecke$ and put $h_0:=I$.

For each $n \in \N$, $n\geq 2$, we have
\begin{equation}\label{recurrence}
\begin{split}
h h_n
 = & \sum_{s \in S}   \sum_{\vert w \vert =n, \vert sw \vert > \vert w \vert}   T_s T_w   +  \sum_{s \in S}   \sum_{\vert w \vert =n, \vert sw \vert < \vert w \vert}   T_s T_w \\
 =& \sum_{s \in S}   \sum_{\vert w \vert =n , \vert sw \vert > \vert w \vert}  T_{sw}  +  \sum_{s \in S}   \sum_{\vert w \vert =n , \vert sw \vert < \vert w \vert}   T_{sw}  + \sum_{s \in S}   \sum_{\vert w \vert =n , \vert sw \vert < \vert w \vert}  p T_{w} \\
  =&    h_{n+1} +  (L-1)  h_{n-1}  + p h_{n}.
\end{split}
\end{equation}
We also have $h^2 = h_2 + p h + L h_0$.
This shows in particular that the algebra generated by $h$ consists of radial operators. 
Moreover viewing the above as a recurrence formula we see that each $h_n$ can be expressed as a polynomial in $h$ and $I$, so that the subspace $A$ generated by $\{h_n: n \in \N\}$ coincides with the unital $^*$-algebra generated by $h$.

Further define the radial subspace $\ell^2(W)_r:=\{(c_w)_{w \in W} \in \ell^2(W): \forall_{w,v \in W, |w|=|v|}\, c_v = c_w\}$ and denote the orthogonal projection from $\ell^2(W)$ onto $\ell^2(W)_r$ by $P_r$. It is easy to see that $A\Omega$ is norm dense in $\ell^2(W)_r$. Thus the unique trace-preserving conditional expectation $\mathbb{E}$ onto $A''\subset \VNHecke$ is given by the formula
\[ \mathbb{E}(T) \Omega = P_r T\Omega, \;\;\; T \in \VNHecke.\]
This shows that the set of radial operators in $\VNHecke$ coincides with $A''$
and passing now to ultraweak closures  we see that $h$ generates the von Neumann algebra of all radial operators.
\end{proof}

Note that the above fact is not true (even for $p=0$) for a general right-angled Coxeter group.
Also note that formulae as \eqref{recurrence} (and the subsequent line in the proof) play a very relevant role in our proof of singularity in Section \ref{Sect=Singular}.

The first main theorem of this paper is based on the idea of Pytlik for the radial algebra in $\VN(\mathbb{F}_n)$ (\cite{Pytlik}; see also \cite{SinclairSmith}). By $R_h \in \VNHecke^r$ we understand the operator on $\ell^2(W)$ given by the \emph{right} action of $\sum_{s\in S} T_s$.

\begin{lem}\label{Lem=Eta}
For every $v,w \in W$ with $\vert v \vert = \vert w \vert$ and for every $\epsilon > 0$ there exists a vector $\eta \in \ell^2(W)$ such that
\[
\Vert e_v - e_w - (h \eta - R_h\eta  ) \Vert_2 < \varepsilon.
\]
\end{lem}
\begin{proof}
We first assume that $w = az$ and $v = zb$ for some word $z \in W$ with $\vert z \vert = \vert v \vert -1$ and some letters $a, b \in S$. In the proof $x$ and $y$ will always be words in $W$ and summations are always over $x$ and $y$. Put for $k \in \mathbb{N}$
\[
\psi_k = \sum_{ \vert x \vert = \vert y \vert = k, \vert x a \vert = \vert b y \vert = k+1}  e_{x  azb y} \in \ell^2(W),
\]
and define also $\psi_0=e_{azb}$. Let $\delta > 0$. As for each $k \in \N$ there are $L(L-1)^{k-1}$ reduced words in $W$ of length $k$,
\begin{equation}\label{Eqn=PsiKEstimate}
 \Vert \left( \frac{1-\delta}{L-1}  \right)^k  \psi_k \Vert_2^2 \leq    \left( \frac{1-\delta  }{L-1  }\right)^{2k} (L-1)^{2k-2}L^2  \leq 4 (1- \delta)^{2k}.
  \end{equation}
This means that we can define
\[
\eta_{\delta} = \sum_{k = 0}^{\infty}  \left( \frac{1-\delta}{L-1}  \right)^k \psi_k \in \ell^2(W).
\]
We claim that the vector $\eta_{\delta}$, for $\delta$ small enough (dependent on $\epsilon$) satisfies the condition of the lemma. To show that we need to analyse the actions of $h$ and $R_h$ on $\psi_k$.
For $k \geq 1$ we have (the bracket term included; the brackets are there in order to define further vectors in the remainder of the proof)
\begin{equation}\label{Eqn=LeftAction}
  \begin{split}
  h \psi_k = &
  \sum_{s\in S} \sum_{\vert x \vert = \vert y \vert = k, \vert x a \vert = \vert b y \vert = k+1, \vert s x \vert = k+1}  e_{  s x  azb y}\\
   &  +  \sum_{s\in S} \sum_{\vert x \vert = \vert y \vert = k, \vert x a \vert = \vert b y \vert = k+1, \vert s x \vert = k-1}     e_{   s x  azb y} \left( + p e_{   x  azb y} \right).
    \end{split}
  \end{equation}
  and similarly, for $k \geq 1$,
   \begin{equation}\label{Eqn=RightAction}
  \begin{split}
R_h  \psi_k = &
 \sum_{s\in S} \sum_{\vert x \vert = \vert y \vert = k, \vert x a \vert = \vert b y \vert = k+1, \vert y s   \vert = k+1}   e_{   x  azb y s}\\
   &  +  \sum_{s\in S} \sum_{\vert x \vert = \vert y \vert = k, \vert x a \vert = \vert b y \vert = k+1, \vert ys \vert = k-1}    e_{     x  azb y s} \left( + p e_{   x  azb y} \right).
    \end{split}
  \end{equation}
Finally
  \begin{equation}\label{Eqn=HonZeroAction}
h \psi_0 = e_{zb} + pe_{azb} + \sum_{s \in S\setminus\{a\}} e_{sazb}, \qquad   R_h \psi_0  = e_{az} + p e_{azb} + \sum_{s \in S\setminus\{b\}} e_{azbs}.
  \end{equation}
We now analyze the `commutators' $h \psi_k - R_h\psi_k $ and their sum.
Note first that for each $k\in \N_0$ the summand in $h \psi_k$   given by  $p e_{   x  azb y}$ also occurs in $R_h \psi_k $.

We define (compare to \eqref{Eqn=HonZeroAction}),
  \[
  \phi_{1,0} = \sum_{s \in S\setminus\{a\}} e_{sazb}, \:\: \phi_{2,0} = e_{zb},  \:\: \chi_{1,0} =  \sum_{s \in S\setminus\{b\}} e_{azbs}, \chi_{2,0} = e_{az}.
  \]
For $k \geq 1$ we set the following notation: let $\phi_{1,k}$ and $\phi_{2,k}$ be the two large sums on respectively the first and second line of \eqref{Eqn=LeftAction}, without the vectors between brackets.  Similarly we define $\chi_{1,k}$ and $\chi_{2,k}$ to be the two large sums on respectively the first and second line of \eqref{Eqn=RightAction}, without the vectors between brackets.


Then we have for all $k \in \N_0$
   \[
     \phi_{1,k}  =   \frac{1}{L-1} \chi_{2,k+1}, \qquad      \chi_{1,k}   =  \frac{1 }{L-1} \phi_{2,k+1},
   \]
so that
   \[
        \phi_{1,k}  -  \frac{1-\delta}{L-1} \chi_{2,k+1}          = \delta  \phi_{1,k}, \qquad
          \chi_{1,k}  -  \frac{1-\delta}{L-1} \phi_{2,k+1}          = \delta  \chi_{1,k}.
   \]
Thus a version of the telescopic argument yields the equality
   \[
   \begin{split}
   h \eta_\delta - R_h \eta_\delta = &
     \sum_{k=0}^\infty \left(   \frac{1-\delta}{L-1} \right)^k \left(   \phi_{1,k} + \phi_{2, k} - \chi_{1,k} - \chi_{2,k} \right) \\
     = &  e_{zb} - e_{az} + \delta \left(  \sum_{k=1}^\infty  \left( \frac{1-\delta}{L-1} \right)^k \left(   \phi_{1,k}   - \chi_{1,k}  \right)  \right).
   \end{split}
   \]
As $\delta \searrow 0$ this can be shown via a similar $\ell^2$-counting estimate as above to converge in norm to   $e_{zb} - e_{az} $. From this we conclude the claim.

  For general $v = v_1 \ldots v_n $ and $w = w_1 \ldots w_n$ with $v_n \not = w_1$ the proposition follows from a triangle inequality and an application of the proof above to each pair $w_k \ldots w_n v_1 \ldots v_{k-1}$ and $w_{k+1} \ldots w_n v_1 \ldots v_{k}$. In case $v_n = w_1$ one can apply the above to the pairs $v_k \ldots v_n b w_1 \ldots w_{k-2}$ and $v_{k+1} \ldots v_n b w_1 \ldots w_{k-1}$ for some letter $b \not = v_n$.
\end{proof}

We are ready to formulate the first main result in this section.

\begin{thm}\label{Thm=MASA}
The radial algebra $\cB$ is a masa in $\VNHecke$.
\end{thm}
\begin{proof}
Suppose that $T \in \cB' \cap \VNHecke$ and write $T \Omega= \sum_{u \in W} c_u e_u$. Let $v,w \in W$ with $\vert v \vert = \vert w \vert$, let $\varepsilon > 0$ and let $\eta$ be as in Lemma \ref{Lem=Eta}. Note that as $T$ commutes with $h$ we have  $\langle T \Omega, h\eta - R_h\eta  \rangle = \langle (hT - R_h T ) \Omega, \eta \rangle = \langle T(h-R_h) \Omega, \eta \rangle= 0$.
Then we get
\[
\vert \langle T \Omega,   e_v - e_w \rangle \vert  \leq
\vert \langle T \Omega,  e_v - e_w + h \eta - R_h \eta  \rangle \vert \leq
 \varepsilon.
\]
As  $\varepsilon > 0$ is arbitrary, we see that $c_w = c_v$. Thus $T$ is radial, which is equivalent to the fact that $T \in \cB$ by Proposition \ref{Prop=Radial}.
\end{proof}

\begin{rmk}
The recurrence formula \eqref{recurrence} allows us to compute explicitly the distribution of $h$ with respect to the canonical trace. As the formula \eqref{recurrence} is valid only from $n=2$ we first define `new' $h_0$ as $\frac{L}{\wt{L}}$, where $\wt{L}:=L-1$, so that with respect to the new variables it holds for all $n\in \mathbb{N}$. For simplicity assume that $q\in [\frac{1}{\wt{L}}, 1]$, so that $\VNHecke$ is a (finite) factor. Then the distribution of $h$ is continuous (as $\cB$ is diffuse) and the main result of \cite{CohenTrenholme} implies that the corresponding density is given (up to a normalising factor) by
\[ \frac{\wt{L} \sqrt{4\wt{L} - (x-p)^2}}{\pi \left[-(x-p)^2 - p (2-L)(x-p) +p^2(L-1)+ L^2\right]} dx.\]
Note that for $p=0$ we obtain, as expected, the distribution of the radial element in the group $(\Z_2)^{*L}$ as computed in Theorem 4 of \cite{CohenTrenholme}.

\end{rmk}

\section{The Puk\'anszky invariant  and singularity of the Hecke MASA}\label{Sect=Singular}

The Puk\'anszky invariant $\mathcal{P}(\cA)$ of a masa $\cA \subseteq \cM$  is  determined by the von Neumann algebra generated by all $\cA$-$\cA$ bimodule homomorphisms of $L^2(\cM)$. We refer to \cite{SinclairSmith} for further discussion of $\mathcal{P}(\cA)$. In \cite{PopaScan} Popa showed that the Puk\'anszky  invariant can be used to prove singularity of certain masas (and indeed this was successfully applied by Radulescu \cite{Radulescu} in order to obtain singularity of the radial masa in $\VN(\mathbb{F}_n)$). We will use this strategy in this section, following very closely the proof of \cite{Radulescu}, to show that the Hecke radial masa discussed in Section \ref{Sect=Masa} is singular. In particular we determine its Puk\'anszky invariant.

\vspace{0.3cm}

We need some terminology. Let again $L\geq 3$, $W=(\Z_2)^{*L}$, $q \in [\frac{1}{L-1},1]$ and let  $\cB$ be the radial subalgebra of the factor $\VNHecke$ (shown to be a masa in Theorem \ref{Thm=MASA}).

\begin{dfn}
The Puk\'anszky invariant of $\cB \subseteq \VNHecke$ is defined as the type of the von Neumann algebra  $\langle h, R_h \rangle' \subseteq \cB(\ell^2(W))$, where $h$ and $R_h$ were defined in Section \ref{Sect=Masa}.
\end{dfn}

Next we introduce the necessary notation in order to determine the Puk\'anszky invariant of $\cB \subseteq \VNHecke$. We need to construct certain bases, which are inspired by Radulescu's bases in free group factors (see \cite{Radulescu}). For $l \in \mathbb{N}_0$ let $q_l: \Hecke \rightarrow \Hecke$ be the natural projection onto the span of $\{T_w, \vert w \vert = l\}$. Write $\Heckel = q_l(\Hecke)$.   As before set $h_l = \sum_{\vert w \vert = l} T_w$. We have for $m \geq 1$ (see \eqref{recurrence} and its subsequent line)
\begin{equation}\label{Eqn=ChiProduct}
h_1 h_m = h_m h_1 = h_{m+1} + p h_m + (L_m-1) h_{m-1},
\end{equation}
where $L_m = L$ if $m \geq 2$ and $L_m = L+1$ if $m = 1$.
Let
\[
S_l = {\rm span} \{  q_l( h_1 x), q_l( x h_1) \mid x \in q_{l-1}(\Hecke) \};
\]
in particular $S_1 = \C h_1$. Further for  $ l \in \mathbb{N}, \gamma \in \Heckel$, set
\[
\gamma_{m,n} = q_{m+n+l}(h_m \gamma h_n), \qquad m,n \in \mathbb{N}_0.
\]
We also set $\gamma_{m,n} = 0$ in case $m <0$ or $n <0$.
Finally for $l \in \mathbb{N}$ and $\gamma \in \Heckel \ominus S_l$ set
\[
X_\gamma = \overline{\spa}^{\Vert \: \Vert_2} \{  \gamma_{m,n} \mid m,n \in \mathbb{N}_0 \}\subset \ell^2(W).
\]
The following Lemma \ref{Lem=Blackbox} collects all computational results we need further.  As all the (rather easy) arguments are basically contained in \cite[Lemma 1]{Radulescu} we merely sketch the proof; all other proofs we give in this section will then be self-contained.
\begin{lem}\label{Lem=Blackbox}
 We have the following:
\begin{enumerate}
\item\label{Item=BB1} For $\gamma \in \Heckel, l \geq 1, m \geq 1, n \geq 0$, we have
\[
h_1 \gamma_{m,n} = \gamma_{m+1,n} + p \gamma_{m,n} + (L-1) \gamma_{m-1, n}.
\]
\item\label{Item=BB2} For $\gamma \in \Heckel \ominus S_l, l \geq 2, m \geq 0, n \geq 0$, we have
\[
h_1 \gamma_{m,n} = \gamma_{m+1,n} + p \gamma_{m,n} + (L-1) \gamma_{m-1, n}.
\]
(Note that only the case $m=0$ was not already covered by \eqref{Item=BB1}).
\item \label{Item=BB3} For $\beta \in \HeckeOne \ominus S_1,    n \geq 0$, we have
\[
h_1 \beta_{0,n} = \beta_{1,n} + p \beta_{0,n} -   \beta_{0, n-1}.
\]

\item\label{Item=BB4} \label{Item=BlackboxOne} For $\gamma \in \Heckel, l \geq 1$ we have,
\[
 \begin{split}
 & q_{l+m+n+1}( h_1 h_m \gamma h_n) = q_{l+n+m+1}( h_1  q_{l+m+n}(h_m \gamma h_n)), \qquad \:m,n \in \mathbb{N}, \\
 & q_{l-m-n-1}( h_1 h_m \gamma h_n) = q_{l-m-n-1}( h_1  q_{l-m-n}(h_m \gamma h_n)),  \qquad \: m,n \textrm{ such that } 0\leq m + n \leq l.
\end{split}
\]
\item\label{Item=BB5} For $\gamma \in \Heckel, l \geq 1$ we have $q_l(h_1 q_{l+1}(h_1 \gamma)) = (L-1) \gamma$.

\item\label{Item=BB6} For $\beta \in \HeckeOne \ominus S_1$ we have $q_n(h_1 q_{n+1}(\beta  h_n ) ) = - q_{n}(\beta  h_{n-1} )$.

\item\label{Item=BB7} For all $\gamma \in \Heckel \ominus S_l, l \geq 2, n\in \mathbb{N}, m \geq 1$ we have $q_{l}(  q_{m+n+l}(h_m \gamma h_n) h_{m+n} ) = 0$.

\end{enumerate}
\end{lem}
\begin{proof}
The proofs of \eqref{Item=BB1} -- \eqref{Item=BB2} are easy consequences of \eqref{Eqn=ChiProduct}, see also \cite[Lemma 1 (a) and (b)]{Radulescu}. The proof of \eqref{Item=BB3} is essentially the same as \cite[Lemma 1.(c)]{Radulescu}. \eqref{Item=BB4} is a direct consequence of \eqref{Eqn=ChiProduct}. \eqref{Item=BB5} and \eqref{Item=BB6} follow from \eqref{Item=BB1} and \eqref{Item=BB3} respectively. \eqref{Item=BB7} follows from \eqref{Item=BB1} and \eqref{Item=BB2}.
\end{proof}

The following theorem gives the cornerstone in our computation of the Puk\'anszky invariant. The idea is based on first showing that for suitable $\beta$ and $\gamma$ the mapping $T: X_\beta \rightarrow X_\gamma$ defined by the formula \eqref{Eqn=TmnBasis} is bounded and invertible. 
Then one uses a basis transition to the respective basis $\{ h_m \beta h_n \}_{m,n \in \mathbb{N}} and \{h_m \gamma h_n\}_{ m,n \in \mathbb{N}}$ to show that $T$ is actually a $\cB-\cB$-bimodule map.

\begin{thm}\label{Thm=PukanszkyCore}
Let $l \in \N$, $l \geq2$, let $\beta \in \HeckeOne \ominus S_1$ and let $\gamma \in \Heckel(W) \ominus S_l$. Then the following hold:
\begin{enumerate}
\item \label{Item=PukanszkyOne} There exists a bounded invertible linear map $T: X_\beta \rightarrow X_\gamma$ determined by
\begin{equation}\label{Eqn=TmnBasis}
T: \beta_{m,n} \mapsto \gamma_{m,n} + \gamma_{m-1, n-1}, \qquad m,n \in \mathbb{N}_0.
\end{equation}
\item \label{Item=PukanszkyTwo} We have $X_\beta = \overline{\cB \beta \cB}^{\Vert \: \Vert_2}$ and $X_\gamma = \overline{\cB \gamma \cB}^{\Vert \: \Vert_2}$. Moreover the map $T$ defined by \eqref{Eqn=TmnBasis} agrees with the linear map
    \begin{equation}\label{Eqn=TchiBasis}
    T: h_m \beta h_n \mapsto h_m \gamma h_n, \qquad m,n \in \mathbb{N}_0.
    \end{equation}
\end{enumerate}
\end{thm}

The proof of Theorem \ref{Thm=PukanszkyCore} proceeds through a couple of lemmas, which we prove in two separate subsections.

\subsection{Proof of Theorem \ref{Thm=PukanszkyCore} part \eqref{Item=PukanszkyOne}} The first statement of Theorem \ref{Thm=PukanszkyCore}  is essentially a consequence of the following orthogonality property.

\begin{lem}\label{Lem=Orthogonality}
Let $l \in \N$, $l \geq2$ and let $\beta, \beta' \in \HeckeOne \ominus S_1$, $\gamma \in \Heckel \ominus S_l$, $\gamma' \in \Heckel, l \geq 2$. We have then for each $m,n,m',n' \in \N_0$
\begin{equation}\label{Eqn=BetaOrthogonality}
\langle \beta_{m,n}, \beta'_{m', n'} \rangle = \delta_{m +n, n'+m'}  (L-1)^{m+n - \vert n - n' \vert}  {(-1)^{\vert n - n' \vert} }  \langle \beta, \beta' \rangle;
\end{equation}
similarly,
\begin{equation}\label{Eqn=GammaOrthogonality}
\langle \gamma_{m,n}, \gamma'_{m', n'} \rangle = \delta_{m, m'} \delta_{n, n'} (L-1)^{m+n}   \langle \gamma, \gamma' \rangle.
\end{equation}
\end{lem}
\begin{proof}
Let us first prove \eqref{Eqn=GammaOrthogonality}.
Firstly, as $\gamma_{m,n}$ (resp. $\gamma'_{m', n'})$ is in the range of $q_{m+n+l}$ (resp. $q_{m'+n'+l}$) we must have $m+n = m'+n'$ or else both sides of \eqref{Eqn=GammaOrthogonality}  are non-zero.
We claim that
\begin{equation}\label{Eqn=Ind}
q_l(h_{m'} q_{m+n+l}(h_m \gamma h_n)  h_{n'}) = \delta_{m, m'} \delta_{n, n'} (L-1)^{m+n} \gamma.
\end{equation}
 For $k := m+n = 0$ this is obvious. We proceed by induction on $k$ and assume the assertion for $k-1$. For $k \geq 1$ one of $m$ and $n$ is non-zero and we may assume without loss of generality that $m \not = 0$ (the proof for $n$ can be done in the same way, or one considers the adjoint of \eqref{Eqn=Ind} which interchanges the roles of $m$ and $n$). If the left hand side of \eqref{Eqn=Ind} is non-zero, then we must have that $m'$ is non-zero, because otherwise this expression reads $q_l( q_{m+n+l}(h_m \gamma h_n)  h_{n+m})$ which is zero by Lemma \ref{Lem=Blackbox} \eqref{Item=BB7}.

 Using \eqref{Eqn=ChiProduct} together with the fact that $q_l(h_r q_{m+n+l}(x) h_n)=0$ for every $r < m$ and $x \in \Hecke$ and $q_{m+n+l}(h_s \gamma h_n) = 0$ for $s < m$, we get
\[
  q_l(h_{m'} q_{m+n+l}(h_m \gamma h_n)  h_{n'})
=  q_l(h_{m'-1} h_1 q_{m+n+l}( h_1 h_{m-1} \gamma h_n)  h_{n'}).
\]
Using then Lemma \ref{Lem=Blackbox} \eqref{Item=BB4} and \eqref{Item=BB5}  for the first  two of the following equalities and then the induction hypothesis
yields
\begin{equation}\label{Eqn=InductiveStep}
\begin{split}
 & q_l(h_{m'} q_{m+n+l}(h_m \gamma h_n)  h_{n'})
=  q_l(h_{m'-1} q_{m+n+l-1} (h_1 q_{m+n+l}( h_1 q_{m+n+l-1} (h_{m-1} \gamma h_n))  h_{n'})  \\
= & (L-1) q_l(h_{m'-1}  q_{m+n+l-1}( h_{m-1} \gamma h_n)  h_{n'})
=   (L-1) (L-1)^{m+n-1} \delta_{m, m'}  \delta_{n, n'}   \gamma.
\end{split}
\end{equation}
This completes the proof of \eqref{Eqn=Ind}.
Then  using the fact that $h_{m'}$ and $h_{n'}$ are self-adjoint we get
\begin{equation}\label{Eqn=EndConclusion}
\begin{split}
& \langle \gamma_{m,n}, \gamma'_{m', n'} \rangle
=  \langle q_{m+n+l}(h_m \gamma h_n), q_{m'+n'+l}( h_{m'}  \gamma' h_{n'} ) \rangle
=  \langle h_{m'} q_{m+n+l}(h_m \gamma h_n)  h_{n'} ,    \gamma'  \rangle \\
= & \langle q_l(h_{m'} q_{m+n+l}(h_m \gamma h_n)  h_{n'}) ,  \gamma'  \rangle
= (L-1)^{m+n} \delta_{m, m'} \delta_{n, n'} \langle \gamma,  \gamma'   \rangle.
\end{split}
\end{equation}

Next we sketch the proof of \eqref{Eqn=BetaOrthogonality}; it is largely the same as \eqref{Eqn=GammaOrthogonality}. The claim \eqref{Eqn=Ind} gets replaced by the equality
\begin{equation}\label{Eqn=Ind2}
q_l(h_{m'} q_{m+n+l}(h_m \beta h_n)  h_{n'}) =   (L-1)^{\vert m +n \vert - \vert n -n'\vert}   { (-1)^{\vert n - n' \vert} }  \delta_{m+n, m'+n'}  \beta.
\end{equation}
Again the proof proceeds by induction with respect to $k :=  m+n = m' +n'$. The case $k=0$ is obvious so assume $k \geq 1$.
First assume that both $m, m' \geq 1$. Similar to \eqref{Eqn=InductiveStep} and using the same results from Lemma \ref{Lem=Blackbox} we find that
\begin{equation}\label{Eqn=mnbiggerone}
\begin{split}
&  q_l( h_{m'} q_{m+n+l}(h_m \beta h_n)  h_{n'})
= q_l(  h_{m'-1} h_1 q_{m+n+l}( h_1 h_{m-1} \beta h_n)  h_{n'}) \\
= &   (L-1) q_l(  h_{m'-1} q_{m+n+l-1}( h_{m-1} \beta h_n)  h_{n' -1} )
=   (L-1)^{m+n - \vert n - n'\vert}  { (-1)^{\vert n - n' \vert} }  \delta_{m+n, m'+n'}   \langle \beta, \beta' \rangle.
\end{split}
\end{equation}
The proof of the equality \eqref{Eqn=mnbiggerone} (disregarding the intermediate steps) for the case $n, n' \geq 1$ proceeds in the same manner (or follows by taking adjoints of \eqref{Eqn=mnbiggerone} which swaps the roles of $m,m'$ and $n, n'$).
The only case that remains is then $m = 0$ and $n' = 0$ (again the case $m' = 0$ and $n = 0$ follows by taking adjoints, or by symmetry).  Then $n \geq 1, m'\geq 1$ and using Lemma \ref{Lem=Blackbox} \eqref{Item=BB6} for the second equality and then applying the induction hypothesis we obtain
\[
\begin{split}
&  q_1( h_{m'} q_{n+1}(\beta h_n) )
= q_1(  h_{m'-1} q_{n}( h_1 q_{n+1}(  \beta   h_{n-1} h_1)  )  ) \\
= &   { -}  q_1(  h_{m'-1} q_{n}(   \beta h_{n-1}  )    )
=   (L-1)^{m+n - \vert n - n'\vert} \delta_{m+n, m'+n'}  { (-1)^{\vert n - n' \vert} }    \langle \beta, \beta' \rangle.
\end{split}
\]
Then the lemma follows by replacing $\gamma$ by $\beta$ in \eqref{Eqn=EndConclusion}.

\end{proof}

Recall the elementary fact (see \cite[Lemma 5]{Radulescu} for a proof) that for a real number $a, \vert a \vert <1$ there exist constants $B_a>0$ and $C_a>0$ such that for any $ k \in \mathbb{N}, \lambda_1, \ldots, \lambda_k \in \mathbb{C}$ we have
\begin{equation}\label{Eqn=SquareEquivalence}
B_a \sum_{i=1}^k \vert \lambda_i \vert^2 \leq
\sum_{i=1}^k \lambda_i \overline{\lambda}_j a^{\vert i - j \vert} \leq
C_a \sum_{i=1}^k \vert \lambda_i \vert^2.
\end{equation}

\noindent {\it Proof of Theorem \ref{Thm=PukanszkyCore} (\ref{Item=PukanszkyOne}).} By Lemma \ref{Lem=Orthogonality} and \eqref{Eqn=SquareEquivalence} we see that the assignment $\beta_{m,n} \mapsto \gamma_{m,n}$ extends to a bounded invertible linear mapping $T_0: X_\beta \rightarrow X_\gamma$. By Lemma \ref{Lem=Orthogonality} we see that $S: X_\gamma \mapsto X_\gamma: \gamma_{m,n} \mapsto \gamma_{m-1,n-1}$ is bounded with norm $\Vert S \Vert \leq (L-1)^{-2}$. Therefore $\Id_{X_\gamma} + S$ is bounded and invertible. As the composition $(I + S) \circ T_0$ is bounded and invertible and agrees with \eqref{Eqn=TmnBasis} we are done.\qed

\subsection{Proof of Theorem \ref{Thm=PukanszkyCore} part \eqref{Item=PukanszkyTwo}} The following Lemma \ref{Lem=abcomputations} is the crucial part of the proof of Theorem \ref{Thm=PukanszkyCore} \eqref{Item=PukanszkyTwo}.

\begin{lem}\label{Lem=abcomputations}
Let  $l \geq 2$, $\beta \in \HeckeOne \ominus S_1$ and let $\gamma \in \Heckel \ominus S_l$. For every $m,n \in \mathbb{N}_0$ there exist certain constants $b_{k,j}^{m,n}, c_{k,j}^{m,n} \in \mathbb{R}$, $k=0,\ldots,m$, $j=0,\ldots n$ such that we have the expansions
\begin{equation}\label{Eqn=Decomposition}
h_m \beta h_n = \sum_{k \leq m, j \leq n} b_{k,j}^{m,n} \beta_{k,j}, \qquad
h_m \gamma h_n = \sum_{k \leq m, j \leq n} c_{k,j}^{m,n} \gamma_{k,j}.
\end{equation}
 Moreover, these constants satisfy the following equalities:
\begin{equation}\label{Eqn=CoeffDependence}
c_{k,j}^{m,n} = b_{k,j}^{m,n} + b_{k+1,j+1}^{m,n}, \qquad m,n \in \mathbb{N}, k=0,\ldots, m, j= 0, \ldots, n,
\end{equation}
where $b^{m,n}_{m+1, n+1} = 0$.
\end{lem}
\begin{proof}
If $m = 0$ and $n \in \mathbb{N}$ arbitrary, then the existence of decompositions \eqref{Eqn=Decomposition} is a consequence of Lemma \ref{Lem=Blackbox}. The relation \eqref{Eqn=CoeffDependence} for $m = 0$ becomes $c_{k,j}^{0,n} = b_{k,j}^{0,n}$ which is a rather direct consequence of Lemma \ref{Lem=Blackbox} as well.

The proof proceeds by induction on $m$. Let $L_k = L$ if $k >1$ and let $L_1 = L+1$.  We have by \eqref{Eqn=ChiProduct} and then Lemma \ref{Lem=Blackbox} \eqref{Item=BB1} and \eqref{Item=BB3},
\begin{equation}\label{Eqn=BigComputation}
\begin{split}
& h_m \beta h_n
=  (h_1 - p) h_{m-1} \beta h_n  - (L_{m-1}-1) h_{m-2} \beta h_n  \\
= & (h_1 - p) \sum_{k = 0}^{m-1} \sum_{j = 0}^{n}  b_{k,j}^{m-1, n}  \beta_{k, j} - (L_{m-1}-1)  \sum_{k = 0}^{m-2} \sum_{j = 0}^{n}  b_{k,j}^{m-2, n}  \beta_{k,j} \\
= &  \sum_{k = 0}^{m-1} \sum_{j = 0}^{n}  b_{k,j}^{m-1, n}  (\beta_{k+1, j} + (L -1)  \beta_{k-1, j}) -
\sum_{j = 0}^n   b_{0,j}^{m-1, n}  \beta_{0, j-1}
 - (L_{m-1}-1)  \sum_{k = 0}^{m-2} \sum_{j = 0}^{n}  b_{k,j}^{m-2, n}  \beta_{k,j} \\
= &  \sum_{k = 0}^{m} \sum_{j = 0}^{n} ( b_{k-1,j}^{m-1, n} + (L -1)  b_{k+1,j}^{m-1, n})  \beta_{k, j} -
\sum_{j = 0}^{n-1}   b_{0,j+1}^{m-1, n}  \beta_{0, j}
 - (L_{m-1}-1)  \sum_{k = 0}^{m-2} \sum_{j = 0}^{n}  b_{k,j}^{m-2, n}  \beta_{k,j} \\
\end{split}
\end{equation}
This shows that for all $0 \leq k \leq m, 0 \leq j \leq n$ we obtain
\[
b_{k,j}^{m,n} = b_{k-1, j}^{m-1, n} + (L -1 ) b_{k+1, j}^{m-1, n} - (L_{m-1} -1) b_{k,j}^{m-2, n}
- \delta_{k, 0} b_{0, j+1}^{m-1, n}.
\]
Let $\delta_{k \geq 1}$ be 1 if $k \geq 1$ and 0 otherwise.
We get then
\[
b_{k,j}^{m,n} + b_{k+1,j+1}^{m,n} = \delta_{k \geq 1}   (b_{k-1, j}^{m-1, n} + b_{k, j+1}^{m-1, n})
 + (L -1 ) (b_{k+1, j}^{m-1, n}   +  b_{k+2, j+1}^{m, n+1}    ) - (L_{m-1} -1) (b_{k,j}^{m-2, n}  +  b_{k+1,j+1}^{m-2, n} ).
\]
So that by induction
\begin{equation}\label{Eqn=BTransition}
\begin{split}
b_{k,j}^{m,n} + b_{k+1,j+1}^{m,n} =  &\delta_{k \geq 1}   c_{k-1, j}^{m-1, n}
 + (L -1 ) c_{k+1, j}^{m-1, n}  - (L_{m-1} -1) c_{k,j}^{m-2, n} \\
 = &    c_{k-1, j}^{m-1, n}
 + (L -1 ) c_{k+1, j}^{m-1, n}  - (L_{m-1} -1) c_{k,j}^{m-2, n}.
 \end{split}
\end{equation}
Exactly as we computed \eqref{Eqn=BigComputation} (with the difference that Lemma \ref{Lem=Blackbox} \eqref{Item=BB3} is replaced by Lemma \ref{Lem=Blackbox} \eqref{Item=BB2}) we get the equalities
\[
 h_m \gamma h_n = \sum_{k = 0}^{m+1} \sum_{j = 0}^{n} ( c_{k-1,j}^{m-1, n} + (L -1)  c_{k+1,j}^{m-1, n})  \gamma_{k, j}
 - (L_{m-1}-1)  \sum_{k = 0}^{m-2} \sum_{j = 0}^{n}  c_{k,j}^{m-2, n}  \gamma_{k,j}.
\]
Thus
\[
c^{m,n}_{k,j} = c_{k-1, j}^{m-1, n} + (L -1 ) c_{k+1, j}^{m-1, n} - (L_m -1) c_{k,j}^{m-2, n}.
\]
Combining the above with \eqref{Eqn=BTransition} gives $c^{m,n}_{k,j} = b_{k,j}^{m,n} + b_{k+1,j+1}^{m,n}$ for all $0 \leq k \leq m, 0 \leq j \leq n$.
\end{proof}

\noindent {\it Proof of Theorem \ref{Thm=PukanszkyCore} \eqref{Item=PukanszkyTwo}.} Lemma \ref{Lem=abcomputations} shows that $\cB \gamma \cB \subseteq X_{\gamma}$ and $\cB \beta \cB \subseteq X_\beta$ and hence the inclusions hold  also for the $\Vert \: \Vert_2$-closures. For the converse inclusion proceed by induction: take $h_n \gamma h_m \in \cB \gamma \cB$ and assume that all vectors $h_r \beta h_s$ with $r < n, s \leq m$ are contained in $X_\gamma$ (if $n = 0$ then assume that $r \leq n, s < m$  and consider adjoints, or use a similar induction argument on  $m$). By \eqref{Eqn=ChiProduct} we have
\[
h_n \gamma h_m = (h_1 - p) h_{n-1} \gamma h_m - (L_n-1) h_{n-2} \gamma h_m \in h_1 X_{\gamma} + X_\gamma.
\]
Here again $L_n = L$ if $n \geq 2$ and $L_1 = L+1$.
So it suffices to show that $h_1 X_\gamma \subseteq X_\gamma$, but this is a consequence of Lemma \ref{Lem=Blackbox} \eqref{Item=BB2}. The proof for $\beta$ instead of $\gamma$ is the same but uses Lemma \ref{Lem=Blackbox} \eqref{Item=BB1} and \eqref{Item=BB3} for the latter argument.

The fact that \eqref{Eqn=TchiBasis} agrees with \eqref{Eqn=TmnBasis} is now a direct consequence of Lemma \ref{Lem=abcomputations}. Indeed,
\[
\begin{split}
&  T( h_m \beta h_n) = T\left(   \sum_{k \leq m, j \leq n} b_{k,j}^{m,n} \beta_{k,j}   \right)
=  \sum_{k \leq m, j \leq n} b_{k,j}^{m,n} (\gamma_{k,j} + \gamma_{k-1, j-1}) \\
=  & \sum_{k \leq m, j \leq n} (b_{k,j}^{m,n}  + b_{k+1,j+1}^{m,n}  ) \gamma_{k,j}
=   \sum_{k \leq m, j \leq n} c_{k,j}^{m,n}    \gamma_{k,j} = h_m \gamma h_n.
\end{split}
\]
\qed

\subsection{Consequences of Theorem \ref{Thm=PukanszkyCore}}

Let $\cB_r = \langle R_h \rangle''$ (note that as $\VNHecke$ is in the standard form on $\ell^2(W)$, it is also equal to $J \cB J$, where $J$ is the anti-linear Tomita-Takesaki modular conjugation $\delta_x \mapsto \delta_{x^{-1}}$).
For a vector $\gamma \in \bigcup_{l \in \N_0} \Heckel$ we let $p_\gamma$ be the central support in $(\cB \cup \cB_r)''$ of the vector state $\omega_{\gamma, \gamma}$. The operator $p_\gamma$ is then given by the projection onto the closure of $\cB \gamma \cB$.

\begin{lem}\label{Lem=OrthSupport}
If vectors $\xi, \xi' \in \cup_{l \geq 1} \Heckel \ominus S_l$ are orthogonal then $p_{\xi}$ and $p_{\xi'}$ are orthogonal projections.
\end{lem}
\begin{proof}
Let $\xi \in \Heckel \ominus S_l$ and let $\xi' \in \Heckelprime \ominus S_{l'}$ with $l, l' \geq 1$. If $l = l'$ then the lemma follows directly from Lemma \ref{Lem=Orthogonality}. So assume that $l \not = l'$ and say that $l' \leq l$. It suffices to show that
\begin{equation}\label{Eqn=XiOrth}
\xi'_{r,s} \perp \xi_{m,n} \qquad \textrm{ for every } r,s, m,n \in \mathbb{N}_0.
\end{equation}
If $m+n+l \not = r+s+l'$ this is obvious as then the images of $q_{m+n+l}$ and $q_{r+s+l'}$ are mutually orthogonal. We may then assume $m+n+l  = r+s+l'$, so that $r+s \geq m+n$. If $m+n = 0$ then \eqref{Eqn=XiOrth} is obvious, as $\xi \perp S_l$ whereas $\xi'_{r,s} \in S_l$. But then note that $\xi'_{r,s} =  (\xi'_{a, b} )_{r-a, s-b}$ for any $a =0,\ldots,r,$ $b=0, \ldots s$ such that $l'+a+b = l$. As $\xi'_{a,b} \in S_l$ we see from Lemma \ref{Lem=Orthogonality}   that $(\xi'_{a, b} )_{r-a, s-b} \perp \xi_{m,n}$.
\end{proof}

We can now state and prove the main result of this section.

\begin{thm}\label{Thm=Pukanszky}
The von Neumann algebra $(\cB \cup \cB_r)' (1 - p_\Omega)$ is homogeneous of type I$_\infty$.
\end{thm}
\begin{proof}
Because $(\cB \cup \cB_r)''$ is abelian the commutant  $(\cB \cup \cB_r)'$ is necessarily of type I; moreover $(\cB \cup \cB_r)'$ is a direct integral of type I factors (see \cite{DixmierBook} for direct integration). Let $(\xi_i)_{i \in \N}$ be an orthonormal basis in $\cup_{l \geq 1} \Heckel \ominus S_l$. By Lemma \ref{Lem=OrthSupport} the projections $(p_{\xi_i})_{i\in \N}$ are mutually orthogonal and by Theorem \ref{Thm=PukanszkyCore} they have the same central support in $(\cB \cup \cB_r)'$. As by Lemma \ref{Lem=OrthSupport} we have $\sum_{i\in \N} p_{\xi_i} = 1 - p_\Omega$ and $1- p_\Omega$ is central  in $(\cB \cup \cB_r)'$ (c.f. \cite[Lemma 3.1]{PopaScan}) we see that the central support of each $p_{\xi_i}$ in $(\cB \cup \cB_r)'$ is $1 - p_\Omega$. This shows that $(1- p_\Omega) (\cB \cup \cB_r)'$ is a direct integral of I$_\infty$-factors (which by definition means that it is homogeneous of type I$_\infty$).
\end{proof}

\begin{rmk}
Theorem \ref{Thm=Pukanszky} is phrased in the literature as follows: the Puk\'anszky invariant of $\cB$ is $\{ \infty \}$. This is because in the $\cB$-$\cB$-bimodule $(1-p_\Omega) L^2(\cM)$, the only factors occuring in the direct integral decomposition of the commutant of $\cB \cup \cB_r$ are infinite (and necessarily of type I).
\end{rmk}

\begin{cor}
The radial subalgebra
$\cB$ is a singular MASA of $\VNHecke$.
\end{cor}
\begin{proof}
This follows from Theorem \ref{Thm=Pukanszky} by \cite[Remark 3.4]{PopaScan}.
\end{proof}

\section{Generator MASAS in $q$-deformed Gaussian von Neumann algebras}\label{SectionqGaussian}
In this section we consider masas in a different deformation of the free group factors, i.e.\ so-called q-Gaussian algebras.

The starting point of the  construction of $q$-Gaussian algebras is a real Hilbert space $\cHR$. We complexify it, obtaining a complex Hilbert space $\cH$, and form an algebraic direct sum $\bigoplus_{n\geqslant 0} \cH^{\otimes n}$, where $\cH^{\otimes 0}=\C$. Following \cite{BozejkoSpeicher} (see that paper for all facts stated below without proofs), we will define an inner product on this space using the parameter $q \in (-1,1)$. For each $n\in \N$ we define an operator $P_q^n: \cH^{\otimes n} \to \cH^{\otimes n}$ by the formula $P_q^n(e_1 \otimes\dots \otimes e_n) = \sum_{\pi \in S_n} q^{i(\pi)} e_{\pi(1)} \otimes \dots \otimes e_{\pi(n)}$, where $e_1,\dots, e_n \in \cH$, $S_n$ is the permutation group on $n$ letters  and $i(\pi)$ denotes the number of inversions in the permutation $\pi$. These operators are strictly positive, so they define an inner product on $\bigoplus_{n\geqslant 0} \cH^{\otimes n}$ -- the Hilbert space that we get after completion is called the $q$-Fock space and is denoted by $\cFq(\cH)$. The direct sum decomposition of the $q$-Fock space allows us to define shift-like operators.
\begin{dfn}
Let $\xi \in \cH$. We define the \textbf{creation operator} $a_{q}^{\ast}(\xi): \cFq(\cH) \to \cFq(\cH)$ by $a_q^{\ast}(\xi) (e_1\otimes\dots\otimes e_n) = \xi \otimes e_1 \otimes\dots \otimes \dots e_n$. The \textbf{annihilation operator} $a_q(\xi):\cFq(\cH) \to \cFq(\cH)$ is defined as the adjoint of $a_q^{\ast}(\xi)$. Using the definition of the $q$-deformed inner product we can find the formula for $a_q(\xi)$: $a_q(\xi)(e_1\otimes\dots\otimes e_n) = \sum_{i=1}^{n} q^{i-1} \langle\xi, e_i\rangle e_1 \otimes\dots \widehat{e_{i}} \dots\otimes e_n$, where $\widehat{e_i}$ means that the factor $e_i$ is omitted. All the above operators extend to bounded operators on $\cFq(\cH)$.
\end{dfn}
Creation and annihilation operators will allow us to define $q$-Gaussian algebras.
\begin{dfn}
Let $\cHR$ be a real Hilbert space and let $\cH$ be its complexification. The von Neumann subalgebra of $\textrm{B}(\cFq(\cH))$ generated by the set $\{a_q^{\ast}(\xi) + a_q(\xi): \xi \in \cHR\}$ is called the \textbf{$q$-Gaussian algebra} associated with $\cHR$ and is denoted by $\Gamma_q(\cHR)$.

The vector $\Omega=1 \in \C \subset \cH^{\otimes 0} \subset \cFq(\cH)$ is called the \textbf{vacuum vector}. It is a cyclic and separating vector for $\Gamma_q(\cHR)$ and the associated vector state $\omega(x):= \langle \Omega, x\Omega \rangle$ is a normal faithful trace on $\Gamma_q(\cHR)$.
\end{dfn}
\begin{rmk}
For $q=0$ the assignment $\cHR \mapsto \Gamma_q(\cHR)$ is precisely Voiculescu's free Gaussian functor. In particular $\Gamma_{0}(\cHR) \simeq \textrm{L}(\mathbb{F}_{\textrm{dim}(\cHR)})$.
\end{rmk}
We will study problems pertaining to conjugacy of masas in the $q$-Gaussian algebras. It is a nice feature of these objects that the orthogonal operators on $\cHR$ give rise to automorphisms of $\Gamma_q(\cHR)$. To introduce these automorphisms, we need to present the \textbf{first quantisation}.
\begin{dfn}
Let $T: \cH \to \cH$ be a contraction. The assignment $ \bigoplus_{k\geqslant 0} \cH^{\otimes k} \ni e_1 \otimes\dots\otimes e_n \mapsto Te_1 \otimes \dots \otimes Te_n \in \bigoplus_{k\geqslant 0} \cH^{\otimes k}$ extends to a contraction $\cFq(T): \cFq(\cH) \to \cFq(\cH)$ and is called the first quantisation of $T$.
\end{dfn}
\begin{rmk}
If $U: \cH \to \cH$ is a unitary then $\cFq(U)$ is also a unitary.
\end{rmk}
To work with $\Gamma_q(\cHR)$ we need a convenient notation for its generators. For any $\xi \in \cHR$ we put $W(\xi):=a_q^{\ast}(\xi) + a_q(\xi)$. If $\eta=\xi_1 + i \xi_2 \in \cH$ then we denote $W(\eta)=W(\xi_1)+i W(\xi_2)$, therefore $W(\eta)$ is complex-linear in $\eta$. Recall that the vacuum vector $\Omega$ is cyclic and separating. One can check that for any vectors $\eta_1,\dots,\eta_n \in \cH$ we have $\eta_1\otimes\dots\otimes \eta_n \in \Gamma_q(\cHR)\Omega$; the unique operator $W(\eta_1\otimes\dots\otimes \eta_n) \in \Gamma_q(\cHR)$ such that $W(\eta_1\otimes\dots\otimes \eta_n)\Omega = \eta_1\otimes\dots\otimes \eta_n$ is called a \textbf{Wick word}. The span of all such operators associated with finite simple tensors forms a strongly dense $\ast$-subalgebra of $\Gamma_q(\cHR)$, which we call the \textbf{algebra of Wick words}. Finally note that similarly to Section 2 we can also consider the `right' version of $\Gamma_q(\cHR)$, generated by the combinations of right creation and annihilation operators, in particular containing the right Wick words, to be denoted $W_r(\xi)$. We are ready to introduce the \textbf{second quantisation}.
\begin{dfn}
Let $\cHR$ be a real Hilbert space and let $\cH$ be its complexification. Suppose that $T:\cH \to \cH$ is a contraction such that $T(\cHR) \subset \cHR$. Then the assignment $\Gamma_q(\cHR) \ni W(\eta_1\otimes\dots\otimes \eta_n) \mapsto W(T\eta_1\otimes\dots\otimes T\eta_n) \in \Gamma_q(\cHR)$, where $\eta_1,\dots,\eta_n \in \cH$, may be extended to a normal, unital, completely positive map on $\Gamma_q(\cHR)$, denoted by $\Gamma_q(T)$.
\end{dfn}
\begin{rmk}
Note that the condition $T(\cHR) \subset \cHR$ is essential, otherwise $\Gamma_q(T)$ would not even preserve the adjoint, let alone be completely positive.
\end{rmk}
We will only deal with automorphisms and, in this construction, they come from orthogonal operators on $\cHR$. If $U: \cHR \to \cHR$ is orthogonal then $\Gamma_q(U)(x) = \cFq(U) x \cFq(U)^{\ast}$, where we still denote by $U$ its canonical unitary extension to $\cH$. It is easy to check that $\Gamma_q(U) W(\xi) = W(U\xi)$. One can verify that none of these automorphisms is inner, besides the identity.

To find candidates for masas, we draw inspiration from the case $q=0$, in which the most basic masas are the so-called generator masas. In our picture they correspond to subalgebras generated by a single element $W(\xi)$, where $\xi \in \cHR$. In \cite{ricard05qfactor} Ricard proved they are also masas in the case of $q$-Gaussian algebras. As an application, he established factoriality of all $q$-Gaussian algebras $\Gamma_q(\cHR)$ with $\textrm{dim}(\cHR) \geqslant 2$. Recently these generator masas were also shown to be singular (\cite{Wen}) and maximally injective \cite{qMaxInjective}.

Using the automorphisms produced by the second quantisation procedure, we can easily show that all these masas are conjugate by an outer automorphism. Indeed, consider masas generated by $W(\xi)$ and $W(\eta)$, where $\xi,\eta \in \cHR$. By rescaling, we may assume that $\|\xi\|=\|\eta\|=1$. Therefore one can find an orthogonal operator $U$ such that $U\xi=\eta$; then $\Gamma_q(U) ((W(\xi))'') = (W(\eta))''$. Our aim now is to show that they  are never conjugate by a unitary.
\subsection{Case of orthogonal vectors}
We first want to deal with the case when $\mathsf{A}:= (W(e_{1}))''$ and $\mathsf{B}:=(W(e_{2}))''$ are masas in $\mathsf{M}:=\Gamma_{q}(\cHR)$ coming from two orthogonal vectors. In the case $q=0$ these masas correspond to two different generator masas of the free group factor. One can prove that these are not unitarily conjugate using Popa's notion of orthogonal pairs of subalgebras (cf. \cite[Corollary 4.3]{Popa}). We will use another technique due to Popa giving a criterion for embedding $\mathsf{A}$ into $\mathsf{B}$ inside $\mathsf{M}$ (in a certain technical sense). We will actually only state the part of the theorem that is useful for us; for the full statement consult \cite[Theorem 2.1 and Corollary 2.3]{PopaIntertwining}.
\begin{prop}[Popa]
Let $\mathsf{A}$ and $\mathsf{B}$ be von Neumann subalgebras of a finite von Neumann algebra $(\mathsf{M},\tau)$. Suppose that there exists a sequence of unitaries $(u_{k})_{k\in \N} \subset \mathcal{U}(\mathsf{A})$ such that for any $x,y \in \mathsf{M}$ we have $\lim_{k\to\infty} \|\mathbb{E}_{\mathsf{B}} (xu_{k} y)\|_{2}=0$, where $\mathbb{E}_{\mathsf{B}}$ is the unique $\tau$-preserving conditional expectation from $\mathsf{M}$ onto $\mathsf{B}$. Then there does not exist a unitary $u \in \mathsf{M}$ such that $u\mathsf{A} u^{\ast} = \mathsf{B}$.
\end{prop}
\begin{rmk}
Note that it suffices to check that $\lim_{k\to\infty} \|\mathbb{E}_{\mathsf{B}} (xu_{k} y)\|_{2}=0$ only for $x,y \in \widetilde{\mathsf{M}}$, where $\widetilde{\mathsf{M}}$ is a strongly dense $\ast$-subalgebra. It follows from Kaplansky's density theorem, because we can approximate in the strong operator topology (in particular in $L^2$) and control the norm of the approximants at the same time.
\end{rmk}
\begin{prop}\label{Prop:orthogonal}
Let $e_1, e_2 \in \cHR$, $\|e_1\|=\|e_2\|=1$, $e_1 \perp e_2$. Set $\mathsf{A}=(W(e_1))''$, $\mathsf{B}=(W(e_2))''$, and $\mathsf{M}=\Gamma_q(\cHR)$. There exists a sequence of unitaries $(u_{k})_{k \in \N} \subset \mathcal{U}(\mathsf{A})$ such that we have $\lim_{k\to\infty} \|\mathbb{E}_{\mathsf{B}}(xu_k y)\|_{2}=0$ for all $x,y \in \widetilde{\mathsf{M}}$, where $\widetilde{\mathsf{M}}$ is the algebra of Wick words.
\end{prop}
\begin{proof}
Let $(e_n)_{n\in\N}$ be an orthonormal basis of $\cHR$. Assume that $x=W(e_{i_{1}}\otimes\dots\otimes e_{i_{n}})$ and $y=W(e_{j_{1}}\otimes\dots\otimes e_{j_m})$; it clearly suffices because the span of such elements is equal to $\widetilde{\mathsf{M}}$. By definition of the trace on $\Gamma_q(\cHR)$ we have $\|\mathbb{E}_{\mathsf{B}} (xu_k y)\|_{2} = \|(\mathbb{E}_{\mathsf{B}}(xu_k y))\Omega\|$. Since the conditional expectation on the level of the Fock space is just the orthogonal projection (denoted $P$) onto the closed linear span of the set $\{e_{2}^{\otimes n}:n\in \N\}$, we get $\|(\mathbb{E}_{\mathsf{B}}(xu_k y))\Omega\|= \|P(x u_k y\Omega)\|$. Note now that as the left and right actions of $y$ on $\Omega$ produce the same result, $e_{j_{1}}\otimes\dots\otimes e_{j_m}$, we can change $y$ to its right version, $W_{r}(e_{j_{1}}\otimes\dots\otimes e_{j_m})$, denoted now by $\widetilde{y}$. Since $\widetilde{y} \in \mathsf{M}'$, we get $\|P(x u_k y\Omega)\|=\|P(x \widetilde{y} u_{k}\Omega)\|$.  We now choose the sequence $(u_k)_{k \in \N}$ -- it is an arbitrary sequence of unitaries in $\mathsf{A}$ such that the corresponding vectors $\eta_k:=u_{k}\Omega$ converge weakly to zero (such a sequence exists, because $\mathsf{A}$ is diffuse). Let $Q_{l}$ be the orthogonal projection from $\mathcal{F}_{q}(\C e_{1})$ onto $\textrm{span}\{e_{1}^{\otimes j}: j \leqslant l\}$. Then for any $l$ the sequence $(Q_{l} \eta_{k})_{k \in \N}$ converges to zero in norm. Therefore to check that $\lim_{k\to\infty} \|P(x\widetilde{y} \eta_k)\|=0$, it suffices to do it for $\eta_k$ replaced by $(\mathds{1}-Q_l)\eta_k$. We now choose $l=n+m$. Therefore any $\eta_k$ consists solely of tensors $e_{1}^{\otimes d}$, where $d\geqslant n+m+1$. Since $x$ can be written as a sum of products of $n$ (in total) creation and annihilation operators and $y$ can be decomposed similarly into products of $m$ creation and annihilation operators, any simple tensor appearing in $x\widetilde{y}(\mathds{1} - Q_{n+m})\eta_{k}$ will contain at least one $e_{1}$. But all such simple tensors are orthogonal to $\cFq(\C e_{2})$, so they are killed by $P$.
\end{proof}
\begin{cor}
There does not exist a unitary $u \in \Gamma_q(\cHR)$ such that $u(W(e_1))''u^{\ast}=(W(e_{2}))''$, where the vectors $e_{1}$ and $e_{2}$ are orthogonal.
\end{cor}
\subsection{General case}
Let us check now if the method used for a pair of orthogonal vectors can be applied in a more general setting. Assume now that $e_{1}$ and $v$ are two unit vectors and write $v=\alpha e_{1} + \beta e_{2}$, where $e_{2} \perp e_{1}$, $\alpha^2+\beta^2=1$, and $\beta \neq 0$. We fix now an orthonormal basis $(e_{n})_{n \in \N}$ of $\cHR$ (if $\cHR$ is finite-dimensional then this should be a finite sequence).
\begin{prop}
The masas $\mathsf{A}:= W(v)''$ and $\mathsf{B}:=(W(e_1))''$ are not unitarily conjugate.
\end{prop}
\begin{proof}
We proceed exactly as in the proof of Proposition \ref{Prop:orthogonal} and also use the same notation; note however that this time $P$ will be the orthogonal projection onto $\overline{\text{span}}\{e_1^{\otimes n}: n\geqslant 0\}$. The only problem is that now we do not have orthogonality. Write $\eta_{k} = \sum_{j\in\N} a_{j}^{(k)} v^{\otimes j}$. We have $\|v^{\otimes j}\|\simeq \left(\frac{1}{\sqrt{1-q}}\right)^{j}$ (cf. \cite[Third displayed formula on page 660]{ricard05qfactor}). Let us compute $v^{\otimes j}$:
\[
v^{\otimes j} = \sum_{k=0}^{j} \alpha^{j-k} \beta^{k} R_{j,k}(e_{1}^{\otimes (j-k)}\otimes e_{2}^{\otimes k}),
\]
where $R_{j,k}(e_{1}^{\otimes (j-k)} \otimes e_2^{\otimes k})$ is equal to the sum of all simple tensors such that $j-k$ factors are equal to $e_{1}$ and $k$ factors are equal to $e_{2}$; there are $\binom{j}{k}$ such simple tensors. Note now that if $k\geqslant n+m+1$ then after applying $x\widetilde{y}$ at least one $e_2$ remains as a factor, so the orthogonal projection $P$ kills it. We conclude that it suffices to perform the summation in the displayed formula above only up to $j\wedge (n+m)$; we call the resulting tensors $\widetilde{v}^{\otimes j}$ and the corresponding $\eta_k$ is dubbed $\widetilde{\eta}_k$. Since $k$ is bounded, the number $\binom{j}{k}$ is polynomial in $j$, so if we get exponential decay of the norm of the individual factors in the sum, the factor $\binom{j}{k}$ does not affect the overall convergence. After neglecting the terms with $k > n+m$, we use the trivial estimate $\|P(x\widetilde{y} \widetilde{\eta}_k)\|\leqslant C \|\widetilde{\eta}_k\|$. The proof will be completed if we show that $\Vert \widetilde{\eta}_k\Vert$ converges to $0$. Note now that the square of the norm of $\widetilde{\eta}_k$ is equal to $\sum_{j\in \N} |a_{j}^{(k)}|^2 \cdot \|\widetilde{v}^{\otimes j}\|^2$. Recall that $\|\eta_k\|\leqslant 1$ and $\|v^{\otimes j}\| \simeq \left(\frac{1}{\sqrt{1-q}}\right)^{j}$, so the coefficients $a_{j}^{(k)}$ satisfy $\sum_{j\in \N}  |a_{j}^{(k)}|^2 \left(\frac{1}{1-q}\right)^{j} \lesssim 1$. It therefore suffices to show that $\lim_{j\to \infty} (1-q)^{j} \|\widetilde{v}^{\otimes j}\|^2=0$, remembering that the vectors $\eta_k$ converge weakly to $0$, so we only care about large $j$. We estimate the norm of $\widetilde{v}^{\otimes j}$ by the triangle inequality:
\[
\|\widetilde{v}^{\otimes j}\| \leqslant \sum_{k=0}^{j \wedge (n+m)} |\alpha|^{j-k} |\beta|^{k} \binom{j}{k} \|e_{1}^{\otimes k}\otimes e_{2}^{j-k}\|.
\]
Since $k$ is bounded, one can easily get an estimate of the form $\|e_{1}^{\otimes k} \otimes e_{2}^{\otimes (j-k)}\| \leqslant C \left(\frac{1}{\sqrt{1-q}}\right)^{j}$ (cf. \cite[Remark 2]{ricard05qfactor}). This yields $\|\widetilde{v}^{\otimes j}\| \leqslant C\left(\frac{1}{\sqrt{1-q}}\right)^{j} |\alpha|^{j} \cdot j^{k}$. It is the inequality that we wanted, i.e. we find out that $(1-q)^{j} \|\widetilde{v}^{\otimes j}\|^2$ is bounded by $C j^{k} |\alpha|^{j}$, which converges to zero very fast, as we assumed that $|\alpha|<1$. This finishes the proof of the proposition.
\end{proof}

\begin{rmk}
The result above exhibits in particular explicitly a continuuum of non-mutually conjugate singular masas in $\Gamma_q(\cHR)$ (in fact the proofs show that they do not even embed into each other inside $\Gamma_q(\cHR)$, see \cite[Theorem 2.1 and Corollary 2.3]{PopaIntertwining}). Very recently Popa showed in \cite{Popanew} the existence of such uncountable families in a broad class of von Neumann algebras.
\end{rmk}


\begin{rmk}Generator masas can be also studied for the so-called mixed q-Gaussians (cf. \cite{Speicher}). They are known to be masas by \cite{AdamSimeng}, and in fact an application of methods of that paper and general results of \cite{bikrammukherjee16qawfactor} show that they are singular, as noted by Simeng Wang. There seems to be however nothing known about the `radial' subalgebra in this more general context.  Is it a masa? Is it isomorphic to a generator one?
\end{rmk}

\subsection*{Acknowledgements}

\noindent For the first author:

\vspace{0.3cm}

\begin{tabular}{ | p{1.0cm} | p{11cm} |}
\hline \:
\includegraphics[width=10mm]{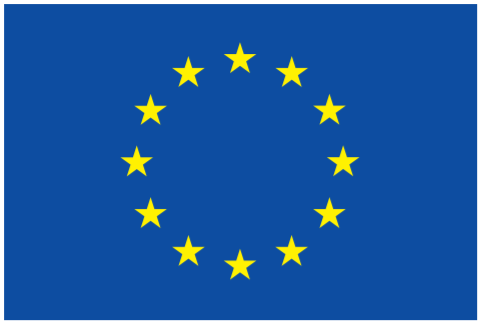}  &
This project has received funding from the European Union's Horizon 2020 research and innovation programme under grant agreement No. 702139. \\ \hline
\end{tabular}

\vspace{0.3cm}

   The second author was partially supported by the National Science Centre (NCN) grant no.~2014/14/E/ST1/00525. The third author was partially supported by the National Science Centre (NCN) grant no.~2016/21/N/ST1/02499. The work on the paper was started during the visit of the first author to IMPAN in January 2017, partially funded by the Warsaw Center of Mathematics and Computer Science.

\end{document}